\documentclass[11pt]{amsproc}

\def\N{{{\Bbb N}}}

\def\R{{\Bbb R}}

\def\l{{\lambda }}
\def\a{{\alpha }}
\def\D{{\Delta }}

\def\a{{\alpha}}
\def\b{{\beta}}
\def\d{{\delta}}
\def\e{{\varepsilon}}

\def\vp{{\varphi}}
\def\t{{\theta }}
\def\g{{\gamma }}
\def\w{{\omega }}
\def\L{{\Lambda }}

\def\L{\mathcal{L}}
\def\){\right)}
\def\({\left(}

\theoremstyle{plain}

\theoremstyle{definition}

\newtheorem{theorem}{Theorem}[section]
\newtheorem{lemma}[theorem]{Lemma}
\newtheorem{remark}[theorem]{Remark}
\newtheorem{corollary}[theorem]{Corollary}
\newtheorem{proposition}[theorem]{Proposition}

\def\N{{{\Bbb N}}}

\def\R{{\Bbb R}}

\def\l{{\lambda }}
\def\a{{\alpha }}
\def\D{{\Delta }}

\def\a{{\alpha}}
\def\b{{\beta}}
\def\d{{\delta}}
\def\e{{\varepsilon}}

\def\vp{{\varphi}}
\def\t{{\theta }}
\def\g{{\gamma }}
\def\w{{\omega }}
\def\L{{\Lambda }}

\def\L{\mathcal{L}}

\def\Dj{\mathcal{\mathcal{D}}}
\def\sw{{\widetilde{\w}}}
\def\st{{\widetilde{\theta}}}
\def\sD{{\widetilde{\D}}}
\def\ab{{(a,b)}}

\def\){\right)}
\def\({\left(}

\par

\begin{document}
\pagespan{1}{}



\title[On approximation by algebraic polynomials in H\"older spaces]{On approximation of functions by algebraic polynomials in H\"older spaces}

%

\author[Yurii
Kolomoitsev]{Yurii
Kolomoitsev$^{\text{a},\text{b},\text{c},*}$}
\address{Institute of Applied Mathematics and Mechanics, NAS of Ukraine, Dobrovol's'kogo~1,
84100 Slov'yans'k, Ukraine}
\email{kolomus1@mail.ru}

\author[Tetiana
Lomako]{Tetiana
Lomako$^{\text{a},\text{b},\text{c}}$}
\address{Institute of Applied Mathematics and Mechanics, NAS of Ukraine, Dobrovol's'kogo~1,
84100 Slov'yans'k, Ukraine}
\email{tlomako@yandex.ru}

\author[J\"urgen~Prestin]{J\"urgen~Prestin$^\text{a}$}
\address{Universit\"at zu L\"ubeck,
Institut f\"ur Mathematik,
Ratzeburger Allee 160,
23562 L\"ubeck}
\email{prestin@math.uni-luebeck.de}

\thanks{$^\text{a}$Universit\"at zu L\"ubeck,
Institut f\"ur Mathematik,
Ratzeburger Allee 160,
23562 L\"ubeck}

\thanks{$^\text{b}$Institute of Mathematics of NAS of Ukraine,
Tereschenkivska 3, 01601 Kiev-4, Ukraine}

\thanks{$^\text{c}$Institute of Applied Mathematics and Mechanics, NAS of Ukraine, Dobrovol's'kogo~1,
84100 Slov'yans'k, Ukraine}


\thanks{$^*$Corresponding author}

\thanks{E-mail address: kolomus1@mail.ru}

\begin{abstract}
  We study approximation of functions by algebraic polynomials in the H\"older spaces corresponding to the generalized Jacobi translation and the Ditzian-Totik moduli of smoothness. By using  modifications of the classical moduli of smoothness, we give improvements of the direct and inverse theorems of approximation and prove the criteria of the precise order of decrease of the best approximation in these spaces.
Moreover, we obtain strong converse inequalities for some methods of approximation of functions. As an example, we consider approximation by the Durrmeyer-Bernstein polynomial operators.
\end{abstract}
\maketitle                   





\section{Introduction}

For a long time, there has been some interest in the approximation of functions in H\"older norms.
This interest originated from the study of a certain class of integro-differential equations and from applications in error estimations for singular integral equations. Following the initial works of Kalandiya~\cite{Ka} and Pr\"ossdorf~\cite{Pr} some problems of approximation in H\"older spaces have been studied by Ioakimidis~\cite{Io}, Bloom and Elliott~\cite{BE}, \cite{E},  Bustamante and Roldan~\cite{BR}, and many others.
One can find an interesting survey on this subject in~\cite{BJ}, see also~\cite{PrPr}.

The first result about approximation in the H\"older spaces~\cite{Ka} was obtained in the case of approximation of functions by algebraic polynomials on an interval. Nevertheless, most interesting and sharp results have been obtained for approximation of periodic functions (see, for example,~\cite{KP},~\cite{LMT} and~\cite{Z}). One of the main reasons of this is the possibility of using some nice properties of the translation operators $f(x)\mapsto f(x+y)$, $x,y\in \mathbb{R}/2\pi \mathbb{Z}$, and well studied methods of Harmonic Analysis on the circle.

In this paper, we study approximation of functions by algebraic polynomials. Unlike the previous investigations, we consider the H\"older spaces generated by the generalized Jacobi translation (see~\cite{WL}). Such approach allows us to apply the well-studied methods of Fourier-Jacobi harmonic analysis (see~\cite{Pl14}) and deal with the problems which were solved earlier only for approximation of functions in the periodic H\"older spaces. In this way, we essentially improve the previously known results and obtain  strong converse inequalities (see~\cite{DI}) for some approximation methods in the H\"older spaces. As an example, we consider approximation by the Durrmeyer-Bernstein polynomial operators.

We also get  several approximation results in the H\"older spaces corresponding to the Ditzian-Totik moduli of smoothness. We will see that these spaces are equivalent to the H\"older spaces corresponding to the generalized Jacobi translation in some sense.
However, in the H\"older spaces corresponding to the Ditzian-Totik moduli of smoothness we can also deal with the case $0<p<1$.

The paper is organized as follows: In Section 2 we introduce the H\"older spaces corresponding to the generalized Jacobi translation and present the auxiliary results related to these spaces. In Section~3 we obtain some new properties of the best approximation and prove analogs of some classical theorems of approximation theory in the H\"older spaces.
In Section~4 we prove the strong converse inequalities for approximation of functions by some linear summation methods of Fourier-Jacobi series. In Section~5 we consider similar problems in the H\"older spaces corresponding to the Ditzian-Totik moduli of smoothness. In Section~6 we improve some estimates of approximation by Bernstein operators in the H\"older spaces.

We denote by $C$ some positive constants which may be different at each occurrence. As usual, $A(f, n)\asymp B(f, n)$ will mean that there exists a positive constant $C$ such that $C^{-1} A(f,n)\le B(f,n)\le C A(f,n)$ for all $f$ and $n$.

\section{The H\"older spaces generated by the generalized Jacobi translation}

Let $a,b>-1$. Denote by
\begin{equation*}
  w(x)=w_{a,b}(x)=(1-x)^a (1+x)^b
\end{equation*}
the Jacobi weight on $[-1,1]$. Let $L_{w,p}=L_p([-1,1]; w)$, $1\le p\le\infty$,
be the space of all functions $f$ measurable on $[-1,1]$ with the finite norm
\begin{equation*}
  \Vert f\Vert_{w,p}=\Vert f\Vert_{L_p([-1,1]; w)}=\(\int_{-1}^1 |f(x)|^p w(x)\mathrm{d}x\)^\frac1p.
\end{equation*}
In the unweighted case, we will write $L_p=L_{p}[-1,1]=L_p([-1,1]; w_{0,0})$, $\Vert f\Vert_p=\Vert f\Vert_{L_p[-1,1]}=\Vert f\Vert_{L_p([-1,1]; w_{0,0})}$. For simplicity, we denote
the space $C[-1,1]$ as $L_\infty[-1,1]$ and
\begin{equation*}
  \Vert f\Vert_\infty=\max_{x\in [-1,1]}|f(x)|.
\end{equation*}

For $a,b>-1$, denote by $P_k^\ab(x)$, $k=0,1,\dots$, the system of Jacobi polynomials orthogonal on $[-1,1]$ such that
\begin{equation*}
P_k^{(a,b)}(1)=\binom{k+a}{k},\quad k=0,1,\dots.
\end{equation*}
Let also $R_k^{(a,b)}$ be normalized Jacobi polynomials,
\begin{equation*}
R_k^{(a,b)}(x)=\frac{P_k^{(a,b)}(x)}{P_k^{(a,b)}(1)},\quad k=0,1,\dots.
\end{equation*}

The expansion of a function $f\in {L_{w,p}}$, $1\le p\le\infty$, $a,b>-1$, in the Fourier-Jacobi series has the form
\begin{equation}\label{eqFJs}
  f(x)\sim\sum_{k=0}^\infty c_k^\ab(f)\mu_k^\ab R_k^\ab(x),
\end{equation}
with Fourier coefficiants
\begin{equation*}
c_k^\ab(f)=\int_{-1}^1 f(x)R_k^\ab(x)w(x)\mathrm{d}x,\quad k=0,1,\dots,
\end{equation*}
and
\begin{equation*}
\mu_k^\ab=\Vert R_k^\ab \Vert_{L_{w,2}}^{-2}\approx k^{2a+1}.
\end{equation*}

The Fourier-Jacobi expansion is closely related to the generalized translation operator
$T_h^\ab$, $0<h<\pi$, acting on a function $f\in L_{w,p}$ with expansion (\ref{eqFJs}) by the following formula
\begin{equation}\label{eqTransTab}
  T_h^{(a,b)} f(x)\sim \sum_{k=0}^\infty c_k^\ab (f)\mu_k^\ab R_k^\ab (\cos h)R_k^\ab(x).
\end{equation}
In particular, if $a=b=0$, then one has
\begin{equation*}
  T_h^{(0,0)}f(x)=\frac1\pi\int_{-1}^1 f\(x\cos h+u\sqrt{(1-x^2)(1-u^2)}\)\frac{\mathrm{d}u}{\sqrt{1-u^2}}.
\end{equation*}

Gasper~\cite{Ga} proved that for $a\ge b\ge -1/2$ and $0<h<\pi$ the operator $T_h^\ab$ is positive. Therefore, we have
\begin{equation}\label{eqOper1}
\Vert T_h^\ab f\Vert_{w,p}\le \Vert f\Vert_{w,p}
\end{equation}
and
\begin{equation}\label{eqOper2}
\Vert f- T_h^\ab f\Vert_{w,p}\to 0\quad\text{as}\quad h\to 0.
\end{equation}

In view of (\ref{eqOper1}) and (\ref{eqOper2}), everywhere below, we will suppose that $a\ge b\ge -1/2$ (see also Remark~\ref{remGasper} below).

Let $r>0$ and $0<h<\pi$. The translation operator (\ref{eqTransTab}) allows us to naturally introduce the modulus of smoothness of the $r$th order by \begin{equation*}
\sw_r(f,h)_{w,p}=\sw_r^{(a,b)}(f,h)_{w,p}=\sup_{0<t\le h} \Vert \sD_t^r f\Vert_{w,p},
\end{equation*}
where
\begin{equation*}
\sD_t^r=\sD_t^{r, \ab} = \(I-T_t^\ab\)^{r/2}=\sum_{k=0}^\infty (-1)^k \binom{r/2}{k} \(T_t^\ab\)^k
\end{equation*}
and
$I$ is the identity operator.

In what follows, we put by definition for $h\ge \pi$
\begin{equation*}
\sw_r(f,h)_{w,p}=\sw_r(f,\pi)_{w,p}=\sup_{0<t<\pi}\Vert \sD_t^r f\Vert_{w,p}.
\end{equation*}

Now we are able to define the H\"older spaces with respect to the generalized Jacobi translation $T_h^\ab$.
We will say that $f\in {H_{w,p}^{r,\a}}$ if $f\in {L_{w,p}}$ and
\begin{equation}\label{eqI4}
  \Vert f\Vert_{H_{w,p}^{r,\a}}=\Vert f\Vert_{w,p}+|f|_{H_{w,p}^{r,\a}}<\infty,
\end{equation}
where
\begin{equation*}
  |f|_{H_{w,p}^{r,\a}}=\sup_{h>0}\frac{\sw_r(f,h)_{w,p}}{h^\a}.
\end{equation*}

Some properties of these spaces can be found in~\cite{WL}.

\subsection{Preliminary remarks and auxiliary results}

%
%

Let $\mathcal{P}_n$ be the set of all algebraic polynomials of degree at most $n$.
As usual, the error of the best approximation of a function $f\in L_{w,p}$ by algebraic polynomials of degree at most $n$ is defined as follows:
\begin{equation*}
E_n(f)_{w,p}=\inf_{P\in \mathcal{P}_{n-1}}\Vert f-P\Vert_{w,p},\quad n\in \N.
\end{equation*}
An algebraic polynomial $P\in \mathcal{P}_{n-1}$  is called a polynomial of the best approximation of $f\in L_{w,p}$ if
\begin{equation*}
\Vert f-P\Vert_{w,p}=E_n(f)_{w,p}.
\end{equation*}

Recall the Jackson-type theorem in $L_{w,p}$, see~\cite{Rus}.
\begin{lemma}\label{lem01}
Let $f\in L_{w,p}$, $1\le p\le\infty$, and $r>0$. Then
\begin{equation*}
  E_n(f)_{w,p}\le C\sw_r\(f,\frac 1n\)_{w,p},\quad n\in\N,
\end{equation*}
where $C$ is a constant independent of $n$ and $f$.
\end{lemma}

The Jacobi polynomials are the eigenfunctions of the differential operator
\begin{equation*}
\Dj=\Dj^{1}_{w}=\frac{-1}{w(x)}\frac{\mathrm{d}}{\mathrm{d}x}w(x)(1-x^2)\frac{\mathrm{d}}{\mathrm{d}x}\,,
\end{equation*}
\begin{equation*}
\Dj P_k^\ab = \l_k^\ab P_k^\ab,\quad \l_k^\ab=k(k+a+b+1).
\end{equation*}

If for $r>0$ and a function $f\in L_{w,p}$, $1\le p\le\infty$,
there exists a function $g\in {L_{w,p}}$
such that its Fourier-Jacobi series has the form
\begin{equation*}
g(x)\sim \sum_{k=1}^\infty \(\l_k^\ab\)^{r/2} c^\ab_k(f) \mu^\ab_k R_k^\ab(x),
\end{equation*}
then we use the notation $g=\Dj^r f$
and call $\Dj^r f$ the (fractional) derivative of order $r$ of the function $f$.

Most results about approximation in $L_{w,p}$ have been formulated in terms of the generalized $K$-functionals
related to the differential operator $\Dj^r$ (see~\cite{Di07}):
\begin{equation*}
\widetilde{K}_r(f,h)_{w,p}=\inf_{g}\left\{\Vert f-g\Vert_{w,p}+h^r \Vert \Dj^r g\Vert_{w,p}\right\}.
\end{equation*}

There is the following natural connection between moduli of smoothness and $K$-functionals (see~\cite{Rus}):
\begin{lemma}\label{lemeqF}
  Let $f\in L_{w,p}$, $1\le p\le\infty$, and $r>0$. Then
  \begin{equation}\label{eqTildeqKw}
\widetilde{K}_r(f,h)_{w,p}\asymp \sw_r(f,h)_{w,p}, \quad 0<h<\pi.
\end{equation}
\end{lemma}

We will often use the following two lemmas. The first one is the Stechkin-Nikolskii type inequality (see~\cite{Rus}):

\begin{lemma}\label{lem1}
Let $1\le p\le\infty$, $n\in\N$, $0<h<\pi/n$, and $r>0$. Then  for any  polynomial $P_n \in \mathcal{P}_n$ we have
\begin{equation*}
  h^{r}\Vert \Dj^r P_n\Vert_{w,p}\asymp \sw_r(P_n,h)_{w,p},
\end{equation*}
where $\asymp$ is a two-sided inequality with absolute constants independent of $P_n$ and $h$.
Moreover, if $P_n \in \mathcal{P}_{n-1}$ is a polynomial of the best approximation of a function $f\in L_{w,p}$, then
\begin{equation*}
\Vert \sD_h^r P_n\Vert_{w,p}\le C\sw_r\(f,\frac1n\)_{w,p},
\end{equation*}
where $C$ is a constant independent of $P_n$, $h$, and $f$.
\end{lemma}

\begin{lemma}\label{lemHolder}
Let $1\le p\le\infty$ and $0<\a<r<k$. Then the norms of a function in $H_{w,p}^{r,\a}$ and $H_{w,p}^{k,\a}$ are equivalent.
%
%
\end{lemma}

\begin{proof}
To prove the lemma we can use the scheme of the proof of  Theorem 10.1 in~\cite[Ch. 2]{DL}. For this we only need to use the inequality
\begin{equation*}
  \sw_k(f,h)_{w,p}\le 2^{k-r}\sw_r(f,h)_{w,p},\quad k>r,\quad h>0,
\end{equation*}
which can be easily obtained from  (\ref{eqOper1}), and the Marchaud inequality
\begin{equation*}
  \sw_r(f,h)_{w,p}\le C h^r \int_h^\pi \frac{\sw_k(f,u)_{w,p}}{u^{r+1}}\mathrm{d}u,\quad k>r,\quad 0<h<\pi.
\end{equation*}
Note that the last inequality follows from the corresponding inequality for the $K$-functionals in \cite{Di98} and Lemma~\ref{lemeqF}.
\end{proof}

\section{Properties of the best approximation. Direct and inverse theorems}

Denote the error of the best approximation in the H\"older space ${H_{w,p}^{r,\a}}$ by
\begin{equation*}
E_n(f)_{H_{w,p}^{r,\a}}=\inf_{P\in \mathcal{P}_{n-1}}\Vert f-P\Vert_{H_{w,p}^{r,\a}},\quad n\in \N.
\end{equation*}
As above, an algebraic polynomial $P\in \mathcal{P}_{n-1}$  is called a polynomial of the best approximation of $f\in H_{w,p}^{r,\a}$ if
\begin{equation*}
\Vert f-P\Vert_{H_{w,p}^{r,\a}}=E_n(f)_{H_{w,p}^{r,\a}}.
\end{equation*}

Let us establish a connection between the errors of the best approximation in the spaces $H_{w,p}^{r,\a}$ and $L_{w,p}$.

\begin{theorem}\label{lem4}
 Let $f\in {H_{w,p}^{r,\a}}$, $1\le p\le \infty$, and $0<\a\le r$. Then
\begin{equation}\label{eqlem4pr1}
  C^{-1} n^\a E_n(f)_{w,p} \le  E_n(f)_{H_{w,p}^{r,\a}}\le C\(n^\a E_n(f)_{w,p}+
\sum_{\nu=n+1}^\infty \nu^{\a-1}E_\nu(f)_{w,p}\),\quad n\in \N,
\end{equation}
where $C$ is a positive constant independent of $n$ and $f$.
\end{theorem}

\begin{proof}
Let $P_n\in \mathcal{P}_{n-1}$, $n\in \N$, be the polynomials of the best approximation of $f\in H_{w,p}^{r,\a}$.
Then by Lemma~\ref{lem01}, we obtain the lower bound by
\begin{equation*}
\begin{split}
n^\a E_n(f)_{w,p} &\le C n^\a\sw_r(f-P_n,1/n)_{w,p} \\
&\le C\sup_{0<h\le 1/n}\frac{\sw_r(f-P_n,h)_{w,p}}{h^{\a}}\le C E_n(f)_{H_{w,p}^{r,\a}}.
\end{split}
 \end{equation*}

Let us prove the upper estimate in (\ref{eqlem4pr1}). Now, let $P_n \in \mathcal{P}_{n-1}$, $n\in\N$,  be the polynomials of the best approximation of $f\in L_{w,p}$.
Let $m\in \N$ such that $2^{m-1}\le n<2^m$.
Assuming that $\sum_{\nu=1}^\infty \nu^{\a-1}E_\nu(f)_{w,p}<\infty$, we can write
\begin{equation}\label{eqPPRREEDD}
  f=P_{2^m}+\sum_{\nu=m}^\infty U_{2^\nu}\quad \text{in}\quad L_{w,p},
\end{equation}
where $U_{2^\nu}=P_{2^{\nu+1}}-P_{2^\nu}$.
From (\ref{eqPPRREEDD}) we have
\begin{equation}\label{eqlem4.1}
  |f-P_n|_{H_{w,p}^{r,\a}}\le |P_{2^m}-P_n|_{H_{w,p}^{r,\a}}+
\sum_{\nu=m}^\infty |P_{2^{\nu+1}}-P_{2^\nu}|_{H_{w,p}^{r,\a}}=S_1+S_2.
\end{equation}
By Lemma~\ref{lem1} and (\ref{eqOper1}), we obtain
\begin{equation}\label{eqlem4.2}
\begin{split}
  S_1&\le \(\sup_{0<h<2^{-m}}+\sup_{h\ge 2^{-m}}\)\frac{\sw_r(P_{2^m}-P_n,h)_{w,p}}{h^{\a}}\\
     &\le C 2^{\a  m}\(\sw_r(P_{2^m}-P_n,2^{-m})_{w,p}+\Vert P_{2^m}-P_n\Vert_{w,p}\)\\
     &\le C 2^{\a  m}\Vert P_{2^m}-P_n\Vert_{w,p}\le C n^{\a}E_n(f)_{w,p}.
\end{split}
\end{equation}
Again, by Lemma~\ref{lem1}  and (\ref{eqOper1}), we get
\begin{equation}\label{eqcor1pr1}
\begin{split}
S_2&\le \sum_{\nu=m}^\infty \sup_{0<h\le 2^{-\nu-1}}h^{-\a } \sw_r(U_{2^\nu},h)_{w,p}+
\sum_{\nu=m}^\infty \sup_{h\ge 2^{-\nu-1}}h^{-\a }\sw_r(U_{2^\nu},h)_{w,p}\\
&\le C\(\sum_{\nu=m}^\infty 2^{\a  \nu}\sw_r(U_{2^\nu},2^{-\nu-1})_{w,p}+
\sum_{\nu=m}^\infty 2^{\a  \nu}\sup_{h\ge 2^{-\nu-1}}\sw_r(U_{2^\nu},h)_{w,p}\)\\
&\le C\sum_{\nu=m}^\infty 2^{\a  \nu}\Vert  U_{2^\nu}\Vert_{w,p}\le C\sum_{\nu=m}^\infty 2^{\a \nu}E_{2^\nu}(f)_{w,p}\le C\sum_{\mu=n}^\infty \mu^{\a -1}E_\mu(f)_{w,p}.
\end{split}
\end{equation}

Thus, combining (\ref{eqlem4.1})--(\ref{eqcor1pr1}), we obtain the upper estimate in (\ref{eqlem4pr1}).
\end{proof}

\begin{corollary}
  Let $f\in {H_{w,p}^{r,\a}}$, $1\le p\le \infty$, $0<\a\le r$, and $\g> 0$. Then the following assertions are equivalent:

  \medskip

  $(i)$ $E_n(f)_{H_{w,p}^{r,\a}}=\mathcal{O}(n^{-\g})$ as $n\to\infty$,

  \medskip

  $(ii)$ $E_n(f)_{w,p}=\mathcal{O}(n^{-\g-\a})$ as $n\to\infty$.
\end{corollary}

By using the upper inequality in (\ref{eqlem4pr1}) and Lemma~\ref{lem01}, one can prove that under the
conditions of Theorem~\ref{lem4} for any $k>0$
\begin{equation}\label{eqJackHold1}
  E_n(f)_{H_{w,p}^{r,\a}}\le C\int_0^{1/n}\frac{\sw_k(f,t)_{w,p}}{t^{\a+1}}\mathrm{d}t,\quad n\in\N,
\end{equation}
where $C$ is a constant independent of $f$ and $n$ (see also Theorem~\ref{th3dt} below).

In the case $\a<r$, one can obtain a sharper estimate by using
the following modulus of smoothness
\begin{equation*}
\tilde{\theta}_{k,\a}(f,t)_{w,p}=\sup_{0<h\le t}\frac{\sw_k(f,h)_{w,p}}{h^\a},\quad 0<\a\le k.
\end{equation*}
The similar moduli of smoothness have initially been used for the investigation of approximation in H\"older spaces (see, for example,~\cite{BJ} and~\cite{BR}).

We prove the following Jackson-type  theorem in terms of $\widetilde{\theta}_{k,\a}(f,h)_{w,p}$.

\begin{theorem}\label{cor2}
   Let $f\in {H_{w,p}^{r,\a}}$, $1\le p\le\infty$, $0<\a< \min(r,k)$ or $0<\alpha=k=r$.
Then
\begin{equation}\label{eqR2}
  E_n(f)_{H_{w,p}^{r,\a}}\le C \widetilde{\theta}_{k,\a}\(f,\frac1n\)_{w,p},\quad n\in\N,
\end{equation}
where $C$ is a constant independent of $n$ and $f$.
\end{theorem}

\begin{proof}
First let $\a<\min(r,k)$ and $P_n\in \mathcal{P}_{n-1}$, $n\in \N$, be polynomials of the best approximation of $f\in L_{w,p}$.
By Lemma~\ref{lem01} and Lemma~\ref{lemHolder},  it suffices to find an estimation for $|f-P_n|_{H_{w,p}^{k,\a}}$.
We have
\begin{equation}\label{eqcor2.1}
  |f-P_n|_{H_{w,p}^{k,\a}}\le \(\sup_{0<h<1/n}+\sup_{h\ge 1/n}\)\frac{\sw_k(f-P_n,h)_{w,p}}{h^\a}=S_1+S_2.
\end{equation}
By (\ref{eqOper1}) and Lemma~\ref{lem01}, we get
\begin{equation}\label{eqcor2.2}
  S_2\le Cn^\a \Vert f-P_n\Vert_{w,p}\le C n^\a\sw_k(f,1/n)_{w,p}\le C\widetilde{\theta}_{k,\a}(f,1/n)_{w,p}.
\end{equation}
We also obtain
\begin{equation}\label{eqcor2.3}
 \begin{split}
   S_1&\le \sup_{0<h<1/n}\frac{\sw_k(f,h)_{w,p}}{h^\a}+\sup_{0<h<1/n}\frac{\sw_k(P_n,h)_{w,p}}{h^\a}\\
   &\le \widetilde{\theta}_{k,\a}(f,1/n)_{w,p}+\sup_{0<h<1/n}\frac{\sw_k(P_n,h)_{w,p}}{h^\a}.
 \end{split}
\end{equation}
To estimate the last term in (\ref{eqcor2.3}) we use Lemma~\ref{lem1}, (\ref{eqOper1}), and Lemma~\ref{lem01}:
\begin{equation}\label{eqcor2.4}
\begin{split}
  \sup_{0<h<1/n}\frac{\sw_k(P_n,h)_{w,p}}{h^\a}&\le C n^\a \sw_k(P_n,1/n)_{w,p}\le C n^\a\(\Vert f-P_n\Vert_{w,p}+
  \sw_k(f,1/n)_{w,p}\)\\
&\le C n^\a \sw_k(f,1/n)_{w,p}\le n^\a \widetilde{\theta}_{k,\a}(f,1/n)_{w,p}.
\end{split}
\end{equation}

Thus, combining (\ref{eqcor2.1})--(\ref{eqcor2.4}), we prove the theorem in the case $\a<\min(r,k)$.
One can use the same scheme to prove the theorem in the case $k=r=\a$.

\end{proof}

By using the standard scheme we also obtain the following inverse result:

\begin{theorem}\label{thR4}
  Let $f\in {H_{w,p}^{r,\a}}$, $1\le p\le\infty$, $0<\a< \min(r,k)$ or $0<\alpha=k=r$.
  Then
\begin{equation}\label{eqR3}
  \widetilde{\theta}_{k,\a}\(f,\frac 1n\)_{w,p}\le \frac{C}{n^{k-\a}}\sum_{\nu=1}^{n} \nu^{k-\a-1}
  E_\nu(f)_{H_{w,p}^{r,\a}},\quad n\in\N,
\end{equation}
where $C$ is a constant independent of $n$ and $f$.
\end{theorem}

\begin{proof}
As in Theorem~\ref{cor2}, we consider only the case $\a<\min(r,k)$.
Let $P_n\in \mathcal{P}_{n-1}$, $n\in \N$, be the polynomials of the best approximation of $f\in H_{w,p}^{r,\a}$.
For any $m\in \N\cup\{0\}$ we get
\begin{equation}\label{eq.thR4.1}
  \st_{k,\a}(f,1/n)_{w,p}\le \st_{k,\a}(f-P_{2^{m+1}},1/n)_{w,p}+\st_{k,\a}(P_{2^{m+1}},1/n)_{w,p}.
\end{equation}
By the definition of the H\"older space and Lemma~\ref{lemHolder}, we obtain
\begin{equation}\label{eq.thR4.2}
 \st_{k,\a}(f-P_{2^{m+1}},1/n)_{w,p}\le C|f-P_{2^{m+1}}|_{H_{w,p}^{r,\a}}\le CE_{2^{m+1}}(f)_{H_{w,p}^{r,\a}}.
\end{equation}
By Lemma~\ref{lemeqF}, Lemma~\ref{lem1}, and Lemma~\ref{lemHolder}, we gain
\begin{equation}\label{eq.thR4.3}
  \begin{split}
    &\st_{k,\a}(P_{2^{m+1}},1/n)_{w,p}\le C\sup_{0<h\le 1/n}\frac{\widetilde{K}_k(P_{2^{m+1}},h)_{w,p}}{h^\a}\le \frac{C}{n^{k-\a}}\Vert \Dj^k P_{2^{m+1}}\Vert_{w,p}\\
    &\le \frac{C}{n^{k-\a}} \(\Vert \Dj^k P_2-\Dj^k P_1\Vert_{w,p}+\sum_{\nu=1}^m \Vert \Dj^k P_{2^{\nu+1}}-\Dj^k P_{2^\nu}\Vert_{w,p}\)\\
    &\le \frac{C}{n^{k-\a}}\(\sw_k(P_2-P_1,1)_{w,p}+\sum_{\nu=1}^m 2^{\nu k}\sw_k\(P_{2^{\nu+1}}-P_{2^\nu},2^{-(\nu+1)}\)_{w,p}\)\\
    &\le \frac{C}{n^{k-\a}} \( \Vert P_2-P_1\Vert_{H_{w,p}^{r,\a}} +\sum_{\nu=1}^m 2^{\nu (k-\a)}  \Vert P_{2^{\nu+1}}-P_{2^\nu}  \Vert_{H_{w,p}^{r,\a}}  \)  \\
    &\le \frac{C}{n^{k-\a}}\(E_1(f)_{H_{w,p}^{r,\a}}+\sum_{\nu=1}^m 2^{\nu(k-\a)}E_{2^\nu}(f)_{H_{w,p}^{r,\a}} \).
  \end{split}
\end{equation}
Since $2^{\nu(k-\a)}E_{2^\nu}(f)_{H_{w,p}^{r,\a}}\le 2\sum_{\mu=2^{\nu-1}+1}^{2^\nu} \mu^{k-\a-1}E_\mu(f)_{H_{w,p}^{r,\a}}$, $\nu\ge 1$, we obtain from (\ref{eq.thR4.3}) that
\begin{equation}\label{eq.thR4.4}
\st_{k,\a}(P_{2^{m+1}},1/n)_{w,p}\le  \frac{C}{n^{k-\a}} \(E_1(f)_{H_{w,p}^{k,\a}}+\sum_{\mu=2}^{2^m} \mu^{k-\a-1}E_\mu(f)_{H_{w,p}^{r,\a}}\).
\end{equation}
Choose $m$ such that $2^m\le n<2^{m+1}$. Then (\ref{eq.thR4.1}), (\ref{eq.thR4.2}), and (\ref{eq.thR4.4}) yield (\ref{eqR3}).
\end{proof}


Now let us consider the problem concerning the sharp order of decrease of the best approximation in the spaces $H_{w,p}^{r,\a}$.

\begin{theorem}\label{th3KP}
                    Let $f\in {H_{w,p}^{r,\a}}$, $1\le p\le\infty$, $0<\a< r$, and $s\ge \a$.
Then the following assertions are equivalent:

\noindent $(i)$ there exists a constant $L>0$ such that
                     \begin{equation}\label{rathore1P}
                           \widetilde{\t}_{s,\a}\bigg(f,\frac 1n\bigg)_{w,p}\le LE_{n}(f)_{H_{w,p}^{r,\a}},\quad n\in\mathbb{N},
                     \end{equation}
$(ii)$ for  some $k>s$ there exists a constant $M>0$ such that
                     \begin{equation}\label{rathore2P}
                            \widetilde{\t}_{s,\a}(f,h)_{w,p}\le M \widetilde{\t}_{k,\a}(f,h)_{w,p},\quad h>0.
                     \end{equation}
\end{theorem}

\begin{proof}
To prove this theorem we follow the scheme of the proof of the corresponding result in
\cite{R}, see also~\cite[Ch. 4]{TB}.
For this purpose, we need  Theorem~\ref{cor2}, Theorem~\ref{thR4}, and the following inequalities:
\begin{equation}\label{svmod1T}
  \widetilde{\t}_{k,\a}(f,\d)_{w,p}\le 2^{k-r}\widetilde{\t}_{r,\a}(f,\d)_{w,p},\quad k>r>0,\quad \d>0,
\end{equation}
\begin{equation}\label{svmod2T}
  \widetilde{\t}_{k,\a}(f,n\d)_{w,p}\le C n^{k-\a}\widetilde{\t}_{k,\a}(f,\d)_{w,p},\quad n\in \N,\quad k,\d>0,
\end{equation}
where $C$ is a constant independent of $f$, $\d$, and $n$. These inequalities can easily be obtained from inequalities  (\ref{eqOper1}) and (\ref{eqTildeqKw}) (see, for example,~\cite[Ch. 2, \S 7]{DL} and \cite{Pl14}).


Let condition~(\ref{rathore2P}) be satisfied. Then from  (\ref{svmod1T}) and (\ref{svmod2T}) we get
\begin{equation}\label{rathoreqiP}
      \widetilde{\t}_{k,\a}(f,n\d)_{w,p}\le C  M n^{s-\a}\widetilde{\t}_{k,\a}(f,\d)_{w,p},\quad n\in\N,\quad \d>0.
\end{equation}
Therefore, $\widetilde{\t}_{k,\a}(f,\l\d)_{w,p}\le CM(1+\l)^{s-\a}\widetilde{\t}_{k,\a}(f,\d)_{w,p}$ for all $\d,\l>0$.

Let us prove that
\begin{equation}\label{rathInvP}
       \frac{1}{n^{k-\a}}\sum_{\nu=1}^n\nu^{k-\a-1}E_\nu(f)_{H_{w,p}^{r,\a}}\le
       C M \widetilde{\t}_{k,\a}\bigg(f,\frac 1n\bigg)_{w,p}.
\end{equation}
Indeed, by Theorem~\ref{cor2} and inequality (\ref{rathoreqiP}),
we obtain
\begin{equation*}
\begin{split}
      &\frac{1}{n^{k-\a}}\sum_{\nu=1}^n\nu^{k-\a-1}E_\nu(f)_{H_{w,p}^{r,\a}}\le
      \frac{C}{n^{k-\a}}\sum_{\nu=1}^n\nu^{k-\a-1} \widetilde{\t}_{k,\a}\left(f,\frac
      {1}{\nu}\right)_{w,p}\\
      &\le\frac{CM}{n^{k-s}} \widetilde{\t}_{k,\a}\left(f,\frac 1n
      \right)_{w,p}\sum_{\nu=1}^n\nu^{k-s-1}\le
      CM \widetilde{\t}_{k,\a}\left(f,\frac 1n\right)_{w,p}.
\end{split}
\end{equation*}
Next, by Theorem~\ref{thR4} and (\ref{rathInvP}), we get that for all $m, n\in\N$
\begin{equation}\label{eqRathorXXXP}
\begin{split}
       &\widetilde{\t}_{k,\a}\(f,\frac{1}{mn} \)_{w,p}\le
       \frac{C}{(mn)^{k-\a}}\sum_{\nu=1}^{mn}\nu^{k-\a-1}E_{\nu}(f)_{H_{w,p}^{r,\a}}\\
       &=\frac{C}{(mn)^{k-\a}}\(\sum_{\nu=n+1}^{mn}\nu^{k-\a-1}E_{\nu}(f)_{H_{w,p}^{r,\a}}+
      \sum_{\nu=1}^n\nu^{k-\a-1}E_{\nu}(f)_{H_{w,p}^{r,\a}} \)\\
       &\le  C\(\frac{1}{(mn)^{k-\a}}\sum_{\nu=n+1}^{mn}
      \nu^{k-\a-1}E_{\nu}(f)_{H_{w,p}^{r,\a}}^p+\frac{M}{m^{k-\a}}
      \widetilde{\t}_{k,\a}\left(f,\frac 1n\right)_{w,p}\).
\end{split}
\end{equation}
Inequality~(\ref{eqRathorXXXP}) implies that
\begin{equation*}
      \sum_{\nu=n+1}^{mn}\nu^{k-\a-1}E_{\nu}(f)_{H_{w,p}^{r,\a}}\ge
      \frac{(mn)^{k-\a}}{C}\widetilde{\t}_{k,\a}\left(f,\frac{1}{mn}\right)_{w,p}-M n^{k-\a}\widetilde{\t}_{k,\a}\left(f,\frac
     1n\right)_{w,p}
\end{equation*}
from which, by using the monotonicity of $E_n(f)_{H_{w,p}^{r,\a}}$ and
(\ref{rathoreqiP}), we derive
\begin{equation*}
      E_n(f)_{H_{w,p}^{r,\a}}\sum_{\nu=n+1}^{mn}\nu^{k-\a-1}\ge (C m^{k-s}-M)n^{k-\a}
      \widetilde{\t}_{k,\a}\left(f,\frac{1}{n} \right)_{w,p}.
\end{equation*}
Thus,  choosing $m$ appropriately, we can find a positive constant  $C=C_{k,s,M}$ such that
\begin{equation*}
      E_n(f)_{H_{w,p}^{r,\a}}\ge C \widetilde{\t}_{k,\a}\left(f,\frac 1n\right)_{w,p}.
\end{equation*}
From the last inequality and (\ref{rathore2P}) we obtain (\ref{rathore1P}).

The reverse direction is an immediate consequence of Theorem~\ref{cor2}.
\end{proof}


One can observe that in Theorem~\ref{lem4} or in (\ref{eqJackHold1}) the best approximation $E_n(f)_{H_{w,p}^{r,\a}}$ can tend to zero very fast.
But at the same time, if for a function $f\in {L_{w,p}}$, we have
$\widetilde{\t}_{r,\a}(f,\d)_{w,p}=o(\d^{r-\a})$, then $f\equiv \mathrm{const}$ by (\ref{svmod2T}).
Thus, estimates (\ref{eqR2}) and (\ref{eqR3}) are not sharp in some sense because of the failure of $\widetilde{\t}_{r,\a}(f,\d)_{w,p}$ in the case $\a=r$. We introduce a  "modulus of smoothness" which, as we suppose,  will be more natural and useful for approximation  in the H\"older spaces. At least, the idea of this "modulus of smoothness" works for the strong converse inequalities in the  next section.

Let $1\le p\le\infty$, $0<\a\le r$, and $k>0$. Denote
\begin{equation}\label{eqnewmod}
  \widetilde{\psi}_{k,r,\a}(f,\d)_{w,p}= \sup\limits_{0<h\le \d}\frac{\widetilde{\w}_k(\sD_h^r f,\d)_{w,p}}{h^\a}.
\end{equation}

Concerning the properties of (\ref{eqnewmod}) we only mention that for any $f\in {L_{w,p}}$, $0<\a\le r$, and $k>0$
\begin{equation}\label{modmod}
  \widetilde{\t}_{k+r,\a}(f,\d)_{w,p}\le \widetilde{\psi}_{k,r,\a}(f,\d)_{w,p}\le C \widetilde{\t}_{r,\a}(f,\d)_{w,p}
\end{equation}
where $0<\d<\pi$ and the constant $C$ depends only on $k$ and $r$.

Indeed, by the definitions of moduli of smoothness and inequality~\eqref{eqOper1}, we obtain
\begin{equation*}
  \begin{split}
      \widetilde{\t}_{k+r,\a}(f,\d)_{w,p}&=\sup_{0<h\le \d}\sup_{0<t\le h}\frac{\Vert\sD_t^{k+r} f\Vert_{w,p}}{h^\a}=
      \sup_{0<h\le \d}\sup_{0<t\le h}\(\frac th\)^\a\frac{\Vert \sD_t^{k} \sD_t^{r} f\Vert_{w,p}}{t^\a}\\
      &\le \sup_{0<t\le \d}\frac{\Vert \sD_t^{k} \sD_t^{r} f\Vert_{w,p}}{t^\a}\le \sup_{0<t\le \d}\frac{\sup_{0<u\le t}\Vert \sD_u^{k} \sD_t^{r} f\Vert_{w,p}}{t^\a}\\
      &=\sup\limits_{0<t\le \d}\frac{\widetilde{\w}_k(\sD_t^r f,\d)_{w,p}}{t^\a}=\widetilde{\psi}_{k,r,\a}(f,\d)_{w,p}.
   \end{split}
\end{equation*}
At the same time, by~\eqref{eqOper1}
\begin{equation*}
  \sup\limits_{0<h\le \d}\frac{\widetilde{\w}_k(\sD_h^r f,\d)_{w,p}}{h^\a}
  \le C\sup\limits_{0<h\le \d}\frac{\Vert\sD_h^r f\Vert_{w,p}}{h^\a}\le  C\sup\limits_{0<h\le \d}\frac{\widetilde{\w}_r( f,h)_{w,p}}{h^\a}
  =C\widetilde{\t}_{r,\a}(f,\d)_{w,p},
\end{equation*}
which gives the second inequality in~\eqref{modmod}.

%

By using the modulus of smoothness (\ref{eqnewmod}), we obtain the following improvement of Theorem~\ref{cor2} in the case $\a=r$:


\begin{theorem}\label{th3}
   Let $f\in H^{r,\a}_{w,p}$, $1\le p\le \infty$, $0<\a\le r$, and $k>0$. Then
\begin{equation}\label{eqM77}
  E_n(f)_{H_{w,p}^{r,\a}}\le C\widetilde{\psi}_{k,r,\a}\(f,\frac 1n\)_{w,p}, \quad n\in\mathbb{N},
\end{equation}
where $C$ is a constant independent of $f$ and $n$.
\end{theorem}

\begin{proof}

We will need the following de la Vall\'ee Poussin-type means
\begin{equation*}
V_n(f)(x)=\sum_{\nu=0}^n v\(\frac \nu n\)c_\nu^\ab(f)\mu_\nu^\ab R_\nu^\ab(x),
\end{equation*}
where
$v\in C^\infty (\R)$, $v(x)=1$ for $|x|\le 1/2$ and $v(x)=0$ for $|x|\ge 1$. It is well-known (see~\cite{Rus} or~\cite{Pl14}) that there exists a constant $C$ such that for any $f\in L_{w,p}$, $1\le p\le\infty$,
\begin{equation}\label{eqVJ}
  \Vert f-V_n(f)\Vert_{w,p}\le C\widetilde{\w}_{r+k}\(f,\frac 1n\)_{w,p}, \quad n\in \N.
\end{equation}
This yields that
\begin{equation}\label{eqPP}
\begin{split}
  E_n(f)_{H_{w,p}^{r,\a}}&\le \Vert f-V_n(f)\Vert_{w,p}+|f-V_n(f)|_{H_{w,p}^{r,\a}}\\
&\le C\widetilde{\w}_{r+k}\(f,\frac1n\)_{w,p}+|f-V_n(f)|_{H_{w,p}^{r,\a}}.
\end{split}
\end{equation}
It is evident (see the last inequality in (\ref{eqPP2}), below) that one only needs to estimate the second term of the right-hand side in (\ref{eqPP}).
We have
\begin{equation}\label{eqPP1}
\begin{split}
|f-V_n(f)|_{H_{w,p}^{r,\a}}\le \(\sup_{0<h<1/n}+\sup_{h\ge1/n}\)\frac{\sw_r(f-V_n(f),h)_{w,p}}{h^\a}=S_1+S_2.
\end{split}
\end{equation}
By inequalities (\ref{eqVJ}) and (\ref{eqOper1}), we get
\begin{equation}\label{eqPP2}
\begin{split}
S_2&\le C n^\a \Vert f-V_n(f)\Vert_{w,p} \le Cn^\a \widetilde{\w}_{r+k}\(f,\frac1n\)_{w,p}=Cn^\a \sup_{0<\d\le 1/n}\Vert \sD_\d^k\sD_\d^r f\Vert_{w,p}\\
&\le C n^\a \sup_{0<h\le 1/n} \sup_{0<\d\le 1/n}\Vert \sD_\d^k\sD_h^r f\Vert_{w,p}\le C \widetilde{\psi}_{k,r,\a}\(f,\frac1n\)_{w,p}.
\end{split}
\end{equation}
In order to estimate $S_1$ we use the equality
$\sD_h^r V_n(f)=V_n (\sD_h^r f)$
and once again (\ref{eqVJ}). Hence,
\begin{equation}\label{eqPP3}
\begin{split}
S_1\le \sup_{0<h\le 1/n}\frac{\Vert \sD_h^r f-V_n(\sD_h^r f)\Vert_{w,p}}{h^\a}\le C\widetilde{\psi}_{k,r,\a}\(f,\frac1n\)_{w,p}.
\end{split}
\end{equation}
Thus, combining (\ref{eqPP})--(\ref{eqPP3}), we get (\ref{eqM77}).

\end{proof}

The following two theorems can be obtained in the same manner as Theorem~\ref{thR4} and  Theorem~\ref{th3KP} above.
However, we already have non-trivial inequalities in the case $\a=r$.

\begin{theorem}\label{thR4con}
Let $f\in H^{r,\a}_{w,p}$, $1\le p\le \infty$, $0<\a\le r$, and  $k>0$.  Then
\begin{equation*}\label{eqM9}
 \widetilde{\psi}_{k,r,\a}\(f,\frac 1n\)_{w,p}\le \frac{C}{n^{k}}\sum_{\nu=1}^{n} \nu^{k-1} E_\nu(f)_{H_{w,p}^{r,\a}},\quad n\in \N,
\end{equation*}
where $C$ is a constant independent of $f$ and $n$.
\end{theorem}

\begin{theorem}\label{th3K}
                    Let $f\in {H_{w,p}^{r,\a}}$, $1\le p\le\infty$, $0<\a\le r$, and $s>0$.
Then the following assertion are equivalent:

\noindent $(i)$ there exists a constant $L>0$ such that
                     \begin{equation*}\label{rathore1}
                           \widetilde{\psi}_{s,r,\a}\bigg(f,\frac 1n\bigg)_{w,p}\le LE_{n}(f)_{H_{w,p}^{r,\a}},\quad n\in \N,
                     \end{equation*}
$(ii)$ for  some $k>s$ there exists a constant
                      $M>0$ such that
                     \begin{equation*}\label{rathore2}
                            \widetilde{\psi}_{s,r,\a}(f,h)_{w,p}\le M \widetilde{\psi}_{k,r,\a}(f,h)_{w,p},\quad h>0.
                     \end{equation*}
\end{theorem}

\section{Strong converse inequalities in the H\"older spaces $H_{w,p}^{r,\a}$}

To formulate the main theorem  in this section we will need some auxiliary notations. For that purpose
let us introduce the general modulus of smoothness.

We will say that $\w=\w(\cdot,\cdot)_{w,p}\in \Omega_{w,p}=\Omega({L_{w,p}},\R_+)$, $1\le p\le \infty$, if for any $f,g\in {L_{w,p}}$

\medskip

$(i)$
  $\w(f,\d)_{w,p}\le C\Vert f\Vert_{w,p}$, $\d>0$,
\medskip

$(ii)$
  $\w(f+g,\d)_{w,p}\le C\(\w(f,\d)_{w,p}+\w(g,\d)_{w,p}\)$, $\d>0$,

\medskip

\noindent where $C$ is a constant independent of $f$, $g$, and $\d$.

\medskip


As a function $\w$ we can take, for example, the modulus of smoothness $\widetilde{\w}_k(f,\d)_{w,p}$ of arbitrary order $k$, or the corresponding
$K$-functional or its realization (see~\cite{DaDi}).

By using $\w\in \Omega_{w,p}$, we define the  new ''modulus of smoothness'' related to the
H\"older space ${H_{w,p}^{r,\a}}$ by
\begin{equation}\label{eqM4}
  \w(f,\d)_{H_{w,p}^{r,\a}}=\w(f,\d)_{w,p}+\sup_{0<h<\pi}\frac{\w(\sD_h^r f,\d)_{w,p}}{h^\a}.
\end{equation}

Note that, if we take $\w(f,\d)_{w,p}=\Vert f\Vert_{w,p}$,
then (\ref{eqM4}) defines the norm in the H\"older spaces which was introduced in (\ref{eqI4}), see also Lemma~\ref{dopprop} below.
But if $\w(f,\d)_{w,p}=\widetilde{\w}_k(f,\d)_{w,p}$, then formula (\ref{eqM4}) provides the definition of the corresponding modulus of smoothness $\widetilde{\w}_k(f,\d)_{H_{w,p}^{r,\a}}$ in the H\"older spaces ${H_{w,p}^{r,\a}}$.

\begin{remark}\label{remReplace}
It is easy to see that Theorem~\ref{th3}, Theorem~\ref{thR4con}, and Theorem~\ref{th3K} remain valid if we replace the modulus of smoothness
$\widetilde{\psi}_{k,r,\a}\(f,1/n\)_{w,p}$ by $\widetilde{\w}_k(f,\d)_{H_{w,p}^{r,\a}}$ in the corresponding theorem.
\end{remark}

\begin{lemma}\label{dopprop}
  Let $1\le p\le\infty$ and $0<\a\le r$. Then
  \begin{equation}\label{dopeq0}
    \sup_{0<h<\pi}\frac{\Vert \sD_h^r f\Vert_{w,p}}{h^\a}=\sup_{h>0}\frac{\sw_r(f,h)_{w,p}}{h^\a}=|f|_{H_{w,p}^{r,\a}}.
  \end{equation}
\end{lemma}
\begin{proof}
Indeed, it is obvious that
  \begin{equation}\label{dopeq1}
    \sup_{0<h<\pi}\frac{\Vert \sD_h^r f\Vert_{w,p}}{h^\a}\le\sup_{h>0}\frac{\sw_r(f,h)_{w,p}}{h^\a}.
  \end{equation}
Thus, to show~\eqref{dopeq0} we have only to verify the converse inequality.

Let $0<\d\le h<\pi$, then
\begin{equation*}
  \begin{split}
     \Vert \sD_\d^r f\Vert_{w,p}\le \d^\a \sup_{0<t\le \d}\frac{\Vert \sD_t^r f\Vert_{w,p}}{t^\a}\le h^\a \sup_{0<t\le \pi}\frac{\Vert \sD_t^r f\Vert_{w,p}}{t^\a}.
   \end{split}
\end{equation*}
Therefore, for any $0<h<\pi$ we get
\begin{equation}\label{dopeq2}
  \frac{\sw_r(f,h)_{w,p}}{h^\a}\le\sup_{0<h<\pi}\frac{\Vert \sD_h^r f\Vert_{w,p}}{h^\a}.
\end{equation}
By the definition of $\sw_r(f,h)_{w,p}$, we also have
$$
\sup_{h\ge \pi}\frac{\sw_r(f,h)_{w,p}}{h^\a}=\frac{\sw_r(f,\pi)_{w,p}}{\pi^\a}=\sup_{0<h<\pi}\frac{\Vert \sD_h^r f\Vert_{w,p}}{\pi^\a}\le\sup_{0<h<\pi}\frac{\Vert \sD_h^r f\Vert_{w,p}}{h^\a}.
$$
The last inequality together with \eqref{dopeq2} implies
  \begin{equation}\label{dopeq3}
    \sup_{h>0}\frac{\sw_r(f,h)_{w,p}}{h^\a}\le\sup_{0<h<\pi}\frac{\Vert \sD_h^r f\Vert_{w,p}}{h^\a}.
  \end{equation}
Finally, combining~\eqref{dopeq1} and~\eqref{dopeq3}, we obtain~\eqref{dopeq0}.
\end{proof}

Now we are ready to formulate the main result of this section.

\begin{theorem}\label{th2}
Let $1\le p\le \infty$, $0<\a\le r$, and $\w \in \Omega_{w,p}$.
Let $\{\L_{n}\}$ be bounded linear operators in ${L_{w,p}}$ such that
\begin{equation}\label{eqTrans}
  T_h^{(a,b)} \(\L_n (f)\)= \L_n \(T_h^{(a,b)} f\),\quad  f\in L_{w,p},\quad n\in \N,\quad  h\in (0,\pi).
\end{equation}
If the following equivalence holds for any $f\in {L_{w,p}}$
\begin{equation}\label{eqM5_1}
  \Vert f-\L_{n}(f)\Vert_{w,p}\asymp \w\(f,\frac 1n\)_{w,p},\quad n\in \N,
\end{equation}
then we have for any $f\in {H_{w,p}^{r,\a}}$
\begin{equation}\label{eqM6_2}
  \Vert f-\L_{n}(f)\Vert_{H_{w,p}^{r,\a}}\asymp \w\(f,\frac 1n\)_{H_{w,p}^{r,\a}}, \quad n\in \N.
\end{equation}
\end{theorem}

\begin{proof}
In view of (\ref{eqTrans}), for any $f\in L_{w,p}$ one has
\begin{equation}\label{eqth1pr5}
\L_{n} (\sD_h^r f)=\sD_h^r \L_{n}(f).
\end{equation}
Thus, by (\ref{eqM5_1}), (\ref{eqth1pr5}), and (\ref{dopeq0}) we get
\begin{equation*}\label{eqth1pr3}
\begin{split}
  \sup_{0<h<\pi}\frac{\w(\sD_h^r f,1/n)_{w,p}}{h^\a}
&\asymp \sup_{0<h<\pi}\frac{\Vert\sD_h^r f-\L_{n}(\sD_h^r f)\Vert_{w,p}}{h^\a} \\
&= \sup_{0<h<\pi}\frac{\Vert \sD_h^r f-\sD_h^r \L_{n}(f)\Vert_{w,p}}{h^\a}=|f-\L_n(f)|_{H_{w,p}^{r,\a}},
\end{split}
\end{equation*}
which together with (\ref{eqM5_1}) yields (\ref{eqM6_2}).
\end{proof}

\begin{remark}\label{rem+}
  It is easy to see that in the assertion of Theorem~\ref{th2} one can simultaneously replace two-sided inequality~\eqref{eqM5_1} by $\Vert f-\L_{n}(f)\Vert_{w,p}\le C\w\(f,1/n\)_{w,p}$ and two-sided inequality \eqref{eqM6_2} by $\Vert f-\L_{n}(f)\Vert_{H_{w,p}^{r,\a}}\le C\w\(f, 1/n\)_{H_{w,p}^{r,\a}}$. The same is true with the sign ''$\ge$'' in place of ''$\le$''.
\end{remark}

It turns out that, in general, strong converse inequalities  do not hold in terms of $\widetilde{\t}_{r,\a}(f,\d)_{w,p}$.
However, we have the following result:

\begin{proposition}\label{thR5}
Let $1\le p\le \infty$ and $0<\a\le r$.
Suppose  $\{\L_{n}\}$ are bounded polynomial operators from $L_{w,p}$ to $\mathcal{P}_{n-1}$
and the following equivalence holds for any $f\in L_{w,p}$
\begin{equation*}
  \Vert f-\L_{n}(f)\Vert_{w,{p}}\asymp \sw_r\(f,\frac 1n\)_{w,p},\quad n\in \N.
\end{equation*}
Then for any $f\in H_{w,p}^{r,\a}$
\begin{equation}\label{eqR5}
  n^\a\sw_r\(f,\frac1n\)_{w,p}+\Vert f-\L_{n}(f)\Vert_{H_{w,{p}}^{r,\a}}\asymp \widetilde{\t}_{r,\a}\(f,\frac 1n\)_{w,p},\quad n\in \N.
\end{equation}
\end{proposition}

Note that the estimate from above in (\ref{eqR5}) can be obtained by repeating the proof of Theorem~\ref{cor2}. Concerning the estimate from below it turns out that without the first term on the left-hand side of (\ref{eqR5}) these estimates do not hold (see Proposition~\ref{pr3} below).

\begin{proof}
It is sufficient only to prove the estimate from below.
Let $h\in (0,1/n)$ be fixed and $P\in \mathcal{P}_{n-1}$, $n\in\N$, be polynomials of the best approximation in $H_{w,p}^{r,\a}$.
By Lemma~\ref{lem1}, we obtain
\begin{equation*}\label{eqth5pr2}
  \begin{split}
  h^{-\a}\widetilde{\w}_r(f,h)_{w,p}&\le h^{-\a} \(\widetilde{\w}_r(f-P,h)_{w,p}+\widetilde{\w}_r(P,h)_{w,p}\)\\
     &\le h^{-\a} \(\widetilde{\w}_r(f-P,h)_{w,p}+Ch^r\Vert \mathcal{D}^r P\Vert_{w,p}\)\\
     &\le h^{-\a} \widetilde{\w}_r(f-P,h)_{w,p}+C n^{-r+\a}\Vert \mathcal{D}^r P\Vert_{w,p} \\
     &\le \Vert f-P\Vert_{H_{w,p}^{r,\a}}+Cn^\a\widetilde{\w}_r(P,1/n)_{w,p}\\
     &\le \Vert f-P\Vert_{H_{w,p}^{r,\a}}+Cn^\a\widetilde{\w}_r(f,1/n)_{w,p}\\
     &\le \Vert f-\L_{n}(f)\Vert_{H_{w,{p}}^{r,\a}}+Cn^\a\widetilde{\w}_r(f,1/n)_{w,p}.
  \end{split}
\end{equation*}
Proposition~\ref{thR5} is proved.
\end{proof}

Now we show that the first term on the left-hand side of (\ref{eqR5}) cannot be dropped.

\begin{proposition}\label{pr3}
   Let $1\le p\le \infty$ and $0<\a\le r$.
Suppose that $\{\L_{n}\}$ are bounded linear operators in $L_{w,p}$ satisfying (\ref{eqTrans})
and the following inequality holds for any $f\in L_{w,p}$
\begin{equation}\label{eqR6}
  \Vert f-\L_{n}(f)\Vert_{w,p}\le C\sw_r\(f,\frac 1n\)_{w,p},\quad n\in \N,
\end{equation}
where $C$ is some constant independent of $f$ and $n$.
Then for any non-trivial function $f$, i.e. $\sw_r(f,\pi)_{w,p}\neq 0$, $\Dj^r f\in H_{w,p}^{r,\a}$\,, and for any sequence
$\{\e_n\}$ with $\e_n\to 0+$ we have
\begin{equation*}
  \frac{\widetilde{\theta}_{r,\a}(f,1/n)_{w,p}}{\e_n n^\a\sw_r(f,1/n)_{w,p}+\Vert f-\L_{n}(f)\Vert_{H_{w,{p}}^{r,\a}}}\to \infty\quad\text{as}\quad n\to \infty.
\end{equation*}
\end{proposition}

\begin{proof}
Taking into account the inequality $\widetilde{\theta}_{r,\a}(f,1/n)_{w,p}\ge n^\a \sw_r(f,1/n)_{w,p}$,
we only need to prove
\begin{equation*}
  \frac{\Vert f-\L_{n}(f)\Vert_{H_{w,{p}}^{r,\a}}}{n^\a\sw_r(f,1/n)_{w,p}}\to 0 \quad\text{as}\quad n\to \infty.
\end{equation*}

Suppose to the contrary that there exists a constant $C>0$ and a sequence of natural numbers $\{n_k\}$, $n_k\to\infty$ as $k\to\infty$,  such that
\begin{equation}\label{eqpr3pr2}
  n_k^\a   \sw_r(f,1/{n_k})_{w,p}   \le C    \Vert f-\L_{n_k}(f)\Vert_{H_{w,{p}}^{r,\a}}.
\end{equation}

By Lemma~\ref{lemeqF}, we have
\begin{equation}\label{eqpr3pr3}
  \sw_r\(f,\frac1{n}\)_{w,p}\le \frac C{n^r} \Vert \Dj^r f\Vert_{w,p}.
\end{equation}
Hence, by Lemma~\ref{dopprop} and~\eqref{eqpr3pr3}, we obtain
\begin{equation}\label{eqpr3pr4}
\begin{split}
  \sw_r\(f,\frac 1n\)_{H_{w,p}^{r,\a}}&=\sw_r\(f,\frac1n\)_{w,p}+\sup_{0<h<\pi}\frac{\sw_r(\sD_h^r f,\frac1n)_{w,p}}{h^\a}\\
  &\le \frac{C}{n^r}\(\Vert \Dj^r f\Vert_{w,p}+\sup_{0<h<\pi}\frac{\Vert \Dj^r\sD_h^r f\Vert_{w,p}}{h^\a}\)\\
  &= \frac{C}{n^r}\(\Vert \Dj^r f\Vert_{w,p}+\sup_{0<h<\pi}\frac{\Vert \sD_h^r\Dj^r f\Vert_{w,p}}{h^\a}\)\\
  &=\frac{C}{n^r} \Vert \Dj^r f\Vert_{H_{w,p}^{r,\a}}.
\end{split}
\end{equation}
Inequalities (\ref{eqpr3pr4}) together with (\ref{eqR6}) and Remark~\ref{rem+}  imply that
\begin{equation}\label{eqpr3pr44}
  \Vert f-\L_n(f)\Vert_{H_{w,p}^{r,\a}}\le C \sw_r\(f,\frac 1n\)_{H_{w,p}^{r,\a}}\le \frac{C}{n^r} \Vert \Dj^r f\Vert_{H_{w,p}^{r,\a}}.
\end{equation}
Thus, combining (\ref{eqpr3pr2}) and (\ref{eqpr3pr44}), we obtain
\begin{equation}\label{contr+}
  \sw_r\(f,1/{n_k}\)_{w,p}\le C {n_k}^{-r-\a}\Vert \Dj^r f\Vert_{H_{w,p}^{r,\a}}.
\end{equation}
At the same time it is easy to verify (see, e.g.~\cite{Pl14} and \eqref{eqTildeqKw}) that
\begin{equation}\label{contr++}
  \frac{\sw_r\(f,1/{n}\)_{w,p}}{n^r}\le C\sw_r\(f,\frac1n\)_{w,p},\quad n\in \N.
\end{equation}

Finally, combining \eqref{contr+} and \eqref{contr++}, we derive $0<C<1/n_k^\a$ which is a contradiction.


\end{proof}

\textbf{Example.} As an application of Theorem~\ref{th2}, let us consider the Durrmeyer-Bernstein polynomial operators $M_n$, which are denoted for
$f\in L_p[0,1]$, $1\le p\le\infty$, by
\begin{equation*}
M_n(f,x)=\sum_{k=0}^n P_{n,k}(x)(n+1)\int_0^1 P_{n,k}(y)f(y)\mathrm{d}y,
\end{equation*}
where
\begin{equation*}
P_{n,k}(x)=\binom{n}{k}x^k(1-x)^{n-k}.
\end{equation*}

In~\cite{ChDiIv}, it was proved that for any $f\in L_p[0,1]$
\begin{equation*}
  \Vert f-M_n(f)\Vert_{L_p[0,1]}\asymp \inf_g\left\{\Vert f-g\Vert_{L_p[0,1]}+\frac1n \left\Vert  \frac{\mathrm{d}}{\mathrm{d}x}x(1-x)\frac{\mathrm{d}}{\mathrm{d}x}g\right\Vert_{L_p[0,1]}\right\},\quad n\in\N.
\end{equation*}
By Lemma~\ref{lemeqF}, after the affine transform $[-1,1]\mapsto [0,1]$, we get
the following two-sided estimate:
\begin{equation}\label{eqDurBer}
  \Vert f-M_n(f)\Vert_{L_p[0,1]}\asymp \widetilde{\w}_2\(f,\frac1{\sqrt n}\)_{L_p[0,1]},\quad n\in\N,
\end{equation}
where $\widetilde{\w}_2\(f,h\)_{L_p[0,1]}$ is the corresponding modification of $\widetilde{\w}_2\(f,h\)_{p}$ related to the interval $[0,1]$.
Note that the modulus  $\widetilde{\w}_2\(f,h\)_{L_p[0,1]}$ can be computed by the following formula:
\begin{equation*}
  \widetilde{\w}_2\(f,h\)_{L_p[0,1]}=\sup_{0<\d<h}\Vert \bar{{\D}}_\d f\Vert_{L_p[0,1]},
\end{equation*}
where
\begin{equation}\label{eqTranslw2}
  \bar{\D}_h f(x)=f(x)-\frac1\pi\int_{-1}^1 f\(2\sin^2\frac h2-x\cos h-u\sqrt{(3+2x-x^2)(1-u^2)}\)\frac{\mathrm{d}u}{\sqrt{1-u^2}}.
\end{equation}

It is easy to see that the operators $M_n$, $n\in \N$, satisfy condition (\ref{eqTrans}) (see, for example,~\cite[Ch. 10, \S\,8]{DL}).
Thus, by Theorem~\ref{th2} and (\ref{eqDurBer}), we obtain the following strong converse inequality:
\begin{equation}\label{eqMn}
  \Vert f-M_n(f)\Vert_{H_p^{r,\a}[0,1]}\asymp \widetilde{\w}_2\(f,\frac1{\sqrt n}\)_{H_p^{r,\a}[0,1]},\quad n\in\N.
\end{equation}



From inequality (\ref{eqMn}), Theorem~\ref{lem4}, and Remark~\ref{remReplace} one can deduce
\begin{corollary}
  Let $1\le p\le\infty$, $0<\a\le r$, and $0<\g<2$. Then the following conditions are equivalent:

\medskip

  $(i)$  $\Vert f-M_n(f)\Vert_{H_p^{r,\a}[0,1]}=\mathcal{O}\({n^{-\g/2}}\)$, \quad $n\to\infty$,

\medskip

  $(ii)$  $\widetilde{\w}_2\(f,\d\)_{H_p^{r,\a}[0,1]}=\mathcal{O}\(\d^\g\) $, \quad $\d\to 0$,

\medskip

  $(iii)$  $E_{n}(f)_{H_{p}^{r,\a}[0,1]}=\mathcal{O}\({n^{-\g}}\)$, \quad $n\to\infty$,

\medskip

  $(iv)$  $E_{n}(f)_{L_p[0,1]}=\mathcal{O}\({n^{-\g-\a}}\)$, \quad $n\to\infty$.

\end{corollary}

\begin{remark}\label{remGasper}
  If $\a\ge \b>-1$, $\a+\b>-1$, and $k$ and $r$ are even numbers, then  all results of Subsection~2.1, Section~3, and Section~4 remain true  (see~\cite{AW}, \cite{Ga} and the remark in~\cite{Rus}).
\end{remark}

\section{The H\"older spaces with respect to the Ditzian-Totik moduli of smoothness (second approach)}

\subsection{Preliminary remarks and auxiliary results}

In this section, we use another approach to the problem of approximation of functions by algebraic polynomials in the H\"older spaces. The main point of this approach is to consider the H\"older spaces generated by the Ditzian-Totik moduli of smoothness, in which we can deal with the case $0<p<1$.


For simplicity, we consider the unweighted case. However, most of our results are valid with some  Jacobi weights, too.


Now let us introduce the necessary notations. We  denote by $I$  an interval of the real line and by $\vp$ an admissible function with respect to $I$ in the sense of Ditzian-Totik (see~\cite[p. 8]{DT}).
Let $L_p(I)$, $0<p<\infty$, be the usual Lebesgue spaces with the (quasi-)norm
$\Vert f\Vert_p=\(\int_I |f(x)|^p {\rm d}x\)^{1/p}$ and $L_\infty(I)=C(I)$ with the norm $\Vert f\Vert_\infty=\max_{x\in I}|f(x)|$.

For a function $f\in L_p(I)$, $0<p\le\infty$, and $r\in\N$, the Ditzian-Totik modulus of smoothness is given by
\begin{equation*}
\w_r^\vp (f,\d)_p=\sup_{|h|\le \d}\Vert \D_{h\vp}^r f\Vert_p,
\end{equation*}
where $\vp$ is an admissible function for an interval $I$ in the sense of Ditzian-Totik and
\begin{equation*}
\D_{h\vp(x)}^r f(x)=\left\{
                   \begin{array}{ll}
                     \displaystyle \sum_{k=0}^r (-1)^k\binom{r}{k}f\(x+\(\frac r2-k\)h\vp(x)\), & \hbox{$x\pm \frac r2 h\vp(x)\in I$,} \\
                     \displaystyle 0, & \hbox{otherwise.}
                   \end{array}
                 \right.
\end{equation*}

Now we are able to define the H\"older spaces with respect to the Ditzian-Totik modulus of smoothness.
We will say that $f\in H_p^{r,\a,\vp}(I)$, $I\subset \R$, $0<\a\le r$, $r\in\N$, $0<p\le\infty$, if $f\in L_p(I)$ and
\begin{equation*}
  \Vert f\Vert_{H_p^{r,\a,\vp}}=\Vert f\Vert_{H_p^{r,\a,\vp}(I)}=\Vert f\Vert_{L_p(I)}+|f|_{H_p^{r,\a,\vp}(I)}<\infty,
\end{equation*}
where
\begin{equation*}
  |f|_{H_p^{r,\a,\vp}}=|f|_{H_p^{r,\a,\vp}(I)}=\sup_{0<h<1}\frac{\w_r^\vp(f,h)_p}{h^\a}.
\end{equation*}

\begin{proposition}\label{propEq}
  Let $f\in H_p^{r,\a,\vp}[-1,1]$, $1<p<\infty$, $0<\a\le r$,  $r\in \N$, and $\vp(x)=\sqrt{1-x^2}$. Then $f\in H_p^{r,\a}$ and
  \begin{equation*}
   \Vert f\Vert_{H_p^{r,\a,\vp}}\asymp \Vert f\Vert_{H_p^{r,\a}},
  \end{equation*}
  where $H_p^{r,\a}=H_{p,w_{0,0}}^{r,\a}$ is related to the Jacobi translation $T_h^{(0,0)}$.
\end{proposition}

\begin{proof}
From Theorem 7.1 in~\cite{DaDi}, it follows that there exists a constant $C$ such that for any $f\in L_p$ and
$t\in (0,t_0)$
\begin{equation*}
C^{-1} K_r^\vp(f,t)_p\le  \widetilde{K}_r(f,t)_p \le C\(K_r^\vp(f,t)_p+t^r\Vert f\Vert_p\),
\end{equation*}
where $
K_r^\vp(f,t)_p=\inf_g\{\Vert f-g\Vert_p+t^r \Vert \vp^r g^{(r)}\Vert_p\}
$ is the $K$-functional related to a function $\vp$.

It is well-known (see~\cite[Ch. 2]{DT})  that
\begin{equation}\label{eqEqKM}
   K_r^\vp(f,t)_p\asymp \w_r^\vp(f,t)_p,\quad t\in (0,t_0).
\end{equation}
Thus, (\ref{eqTildeqKw}) and (\ref{eqEqKM}) imply Proposition~\ref{propEq}.
\end{proof}

\begin{remark}
  Proposition~\ref{propEq} does not hold for $p=1$ or $p=\infty$ . Nevertheless, for any $0<\a\le 2$ the inequalities
  \begin{equation*}
  \Vert f\Vert_{H_\infty^{2,\a,\vp}}\le C\Vert f\Vert_{H_\infty^{2,\a}}
  \end{equation*}
  and
  \begin{equation*}
  \Vert f\Vert_{H_1^{2,\a}}\le C\Vert f\Vert_{H_1^{2,\a,\vp}}
  \end{equation*}
  are valid (see~\cite[Remark 7.9]{DaDi}).
\end{remark}

Proposition~\ref{propEq} implies that results on approximation in the spaces $H_p^{r,\a}$ can be transferred to $H_p^{r,\a,\vp}$ in the case $1<p<\infty$.   For example, from (\ref{eqMn})  we obtain for $1<p<\infty$ and $\vp(x)=\sqrt{x(1-x)}$ the two-sided estimate
\begin{equation*}
  \Vert f-M_n(f)\Vert_{H_p^{r,\a,\vp}[0,1]}\asymp \widetilde{\w}_2\(f,\frac1{\sqrt{n}}\)_{H_p^{r,\a}[0,1]}.
\end{equation*}
Moreover, taking into account that
\begin{equation*}
\Vert f-M_n(f)\Vert_{L_p[0,1]}\asymp \w_2^\vp\(f,\frac1{\sqrt{n}}\)_{L_p[0,1]}+\frac1n \Vert f\Vert_{L_p[0,1]}
\end{equation*}
(see (8.11) and (8.14) in~\cite{Di07}) we obtain that
\begin{equation*}
  \begin{split}
     \Vert f-M_n(f)\Vert_{H_p^{r,\a,\vp}[0,1]}&\asymp \w_2^\vp\(f,\frac1{\sqrt{n}}\)_{L_p[0,1]}\\&+\sup_{0<h<1} h^{-\a}\(\w_2^\vp\(\bar{\D}_h^r f,\frac1{\sqrt{n}}\)_{L_p[0,1]}+
\frac1n \Vert \bar{\D}_h^r f\Vert_{L_p[0,1]}\),
   \end{split}
\end{equation*}
where $\bar{\D}$ was defined by (\ref{eqTranslw2}).

Below, we collect auxiliary results, which correspond to the similar results from Section~2.1.  In what follows we let $I=[-1,1]$,
$\vp(x)=\sqrt{1-x^2}$, and $p_1=\min(p,1)$.


\begin{lemma}\label{lemHolderdt} {\rm (See~\cite{Li} and~\cite{DT88})}.
Let $0< p\le\infty$, $k,r\in\N$, and $0<\a<r<k$. Then the (quasi-)norms of a function in the spaces $H_{p}^{r,\a,\vp}$ and $H_{p}^{k,\a,\vp}$ are equivalent.
%
%
\end{lemma}

Recall the Jackson-type theorem for the Ditzian-Totik modulus of smoothness in $L_p$-spaces (see in~\cite{DT} for $1\le p\le\infty$ and  in~\cite{DLY} for $0<p<1$).
\begin{lemma}\label{lem01dt}
Let $f\in L_p$, $0< p\le\infty$, and $k\in \N$. Then
\begin{equation*}
  E_n(f)_p\le C\w_k^\vp\(f,\frac 1n\)_p,\quad n>4k,
\end{equation*}
where $C$ is a constant independent of $n$ and $f$.
\end{lemma}

The following Stechkin-Nikolskii type inequality corresponds to Lemma~\ref{lem1} (see~\cite{HuL}).

\begin{lemma}\label{lem1dt}
Let $0<p\le\infty$, $n\in\N$, $0<h\le 1/n$, and $r\in\N$. Then for any algebraic polynomial $P_n \in \mathcal{P}_n$ we have
\begin{equation*}
  h^r\Vert \vp^r P_n^{(r)}\Vert_p\asymp  \w_r^\vp(P_n,h)_p,
\end{equation*}
where $\asymp$ is a two-sided inequality with absolute constants independent of $P_n$ and $h$.
Moreover, if $P_n$ is a polynomial of the best approximation of a function $f\in L_p$, then
\begin{equation*}
\w_r^\vp(P_n,h)_p\le C\w_r^\vp\(f,\frac1n\)_p,
\end{equation*}
where $C$ is a constant independent of $P_n$, $h$, and $f$.
\end{lemma}

\subsection{Properties of the best approximation in $H_p^{r,\a,\vp}$. Direct and inverse theo\-rems}

In this subsection we present results which correspond to the similar results from Section~3.
The proofs can easily be obtained by using the schemes of proofs for the corresponding results from Section 3 and the above auxiliary results.

As above, we denote the error of the best approximation in the H\"older space ${H_{p}^{r,\a,\vp}}$ by
\begin{equation*}
E_n(f)_{H_{p}^{r,\a,\vp}}=\inf_{P\in \mathcal{P}_{n-1}}\Vert f-P\Vert_{H_{p}^{r,\a,\vp}},\quad n\in \N.
\end{equation*}
A polynomial $P\in \mathcal{P}_{n-1}$ is called a polynomial of the best approximation of $f\in H_{p}^{r,\a,\vp}$ if
\begin{equation*}
\Vert f-P\Vert_{H_{p}^{r,\a,\vp}}=E_n(f)_{H_{p}^{r,\a,\vp}}.
\end{equation*}

As above, we have the following connection between the errors of the best approximation in the spaces  ${H_p^{r,\a,\vp}}$ and $L_p$:

\begin{theorem}\label{lem4dt}
 Let $f\in H_p^{r,\a,\vp}$, $0<p\le \infty$, $0<\a\le r$, and $r\in \N$. Then
\begin{equation*}
  C^{-1} n^\a E_n(f)_p \le  E_n(f)_{H_p^{r,\a,\vp}}\le C\(n^\a E_n(f)_p+
\(\sum_{\nu=n}^\infty \nu^{\a {p_1}-1}E_\nu(f)_p^{p_1}\)^\frac1{p_1}\),\quad n\in\N,
\end{equation*}
where $C$ is a positive constant independent of $n$ and $f$.
\end{theorem}

Theorem~\ref{lem4dt} and some standard arguments give us the following direct estimate:

\begin{theorem}\label{th3dt}
  Let $f\in H_p^{r,\a,\vp}$, $0<p\le\infty$, $0<\a\le r$, and $k,r\in\N$. Then
\begin{equation}\label{eqM77Ddt}
E_n(f)_{H_p^{r,\a,\vp}}\le C \(\int_0^{1/n}\(\frac{\w_{k}^\vp(f,t)_p}{t^\a}\)^{p_1}\frac{\mathrm{d}t}{t}\)^\frac1{p_1},\quad n\in\N,
\end{equation}
where $C$ is a constant independent of $f$ and $n$.
\end{theorem}



Some slight improvements of (\ref{eqM77Ddt}) can be obtained by using
the modulus of smoothness
\begin{equation*}
\theta_{k,\a}^\vp(f,\d)_p=\sup_{0<h\le \d}\frac{\w_k^\vp(f,h)_p}{h^\a}.
\end{equation*}

%

\begin{theorem}\label{cor2dt}
   Let $f\in H_p^{r,\a,\vp}$, $0<p\le\infty$, $0<\a< \min(r,k)$ or $0<\alpha=k=r$, and $r,k\in\N$.
Then
\begin{equation*}\label{eqR2dt}
  E_n(f)_{H_p^{r,\a,\vp}}\le C \theta_{k,\a}^\vp\(f,\frac1n\)_p,\quad n>4k,
\end{equation*}
%
%
%
%
%
\begin{equation*}\label{eqR3dt}
  \theta_{k,\a}^\vp\(f,\frac 1n\)_p\le \frac{C}{n^{k-\a}}\(\sum_{\nu=1}^{n} \nu^{(k-\a) {p_1}-1} E_\nu(f)_{H_p^{r,\a,\vp}}^{p_1}\)^\frac1{p_1},\quad n\in\N,
\end{equation*}
where $C$ is a constant independent of $n$ and $f$.
\end{theorem}

Note that Theorem~\ref{cor2dt} was obtained in~\cite{BR} in the case $1\le p\le\infty$ and $\a<k=r$.


Now let us consider the problem of the precise order of decrease of the best approximation in $H_p^{r,\a,\vp}$.
The proof of Theorem~\ref{th3KPdt} below is similar to the proof of Theorem~\ref{th3KP}, see also the proof of related results in the case $0<p<1$ in~\cite{K07}. We only note that instead of (\ref{svmod1T})  we will use the  inequality
\begin{equation*}
  \t_{k,\a}^\vp(f,\d)_p\le C\t_{r,\a}^\vp(f,\d)_p,\quad k>r,\quad \d\in (0,\d_0).
\end{equation*}
This inequality can be obtained from Lemma~\ref{lem01dt}, Lemma~\ref{lem1dt}, and the following Bernstein type inequality (see~\cite{DJL})
\begin{equation*}
\Vert \vp^r P_n'\Vert_p\le C n \Vert \vp^{r-1} P_n\Vert_p,\quad 0<p\le\infty,\quad 1\le r\le n,
\end{equation*}
where a constant $C$ is independent of $n$.
Concerning an analog of inequality (\ref{svmod2T})  we have in the case $1\le p\le \infty$, that
\begin{equation}\label{svmod2Tdt}
  \t_{k,\a}^\vp(f,n\d)_p\le C n^{k-\a}\t_{k,\a}^\vp(f,\d)_p,
\end{equation}
where a constant $C$ is independent of $f$ and $\d\in (0,\d_0)$. Inequality (\ref{svmod2Tdt}) is a simple corollary from the corresponding inequality for Ditzian-Totik moduli of smoothness
\begin{equation*}
  \omega_{{k}}^{\vp}(f,n\d)_p\le C n^{k}\omega_{k}^\vp(f,\d)_p
\end{equation*}
(see, e.g.,~\cite[Ch. 2]{DT}).
In the case $0<p<1$, repeating step by step the proof of Lemma~5.2 in~\cite{DHI} (see also Corollary 5.5 in~\cite{DHI})
and using the equality (see, e.g. \cite[p.~187--188]{PePo})
\begin{equation*}
\Delta_{n\delta\varphi(x)}^{r}f(x)=\sum_{\nu=0}^{r(n-1)}A_{\nu,n}^{(k)}\,\Delta_{\delta\varphi(x)}^{r}f\left(x+\left(\nu-\frac{r(n-1)}{2}\right)\delta\varphi(x)\right),
\end{equation*}
where $0<A_{\nu,n}^{(k)}\leq n^{k-1}$ and $x\pm rn \d \vp(x)/2\in (-1,1)$, we can prove that for every
$f\in L_p$, $0<p<1$, and $k, n\in \N$ one obtains
\begin{equation*}
\omega_{k}^\vp(f,n\d)_p\le C n^{\frac{3}{p}+2(k-1)}\omega_{k}^\vp(f,\d)_p\,.
\end{equation*}
Therefore,  in the case $0<p<1$ we get
\begin{equation*}
\t_{k,\a}^\vp(f,n\d)_p\le C n^{\frac{3}{p}+2(k-1)-\a}\t_{k,\a}^\vp(f,\d)_p\,.
\end{equation*}


\begin{theorem}\label{th3KPdt}
Let $f\in H_p^{r,\a,\vp}$, $0<p\le \infty$, $0<\a< r$, $s\ge \a$, and $r,s\in\N$. Then the following assertions are equivalent:

\noindent $(i)$ there exists a constant $L>0$ such that
                     \begin{equation*}
                           \t_{s,\a}^\vp\bigg(f,\frac 1n\bigg)_{p}\le LE_{n}(f)_{H_p^{r,\a,\vp}},\quad n\in\N,
                     \end{equation*}
$(ii)$ for  some
                    \begin{equation*}
k>\left\{
                   \begin{array}{ll}
                     \displaystyle s, & \hbox{$1\le p\le \infty$,} \\
                     \displaystyle 3/p+2(s-1), & \hbox{$0<p<1$},
                   \end{array}
                 \right.
                  \end{equation*}
there exists a constant $M>0$ such that
                     \begin{equation*}
                            \t_{s,\a}^\vp(f,h)_{p}\le M \t_{k,\a}^\vp(f,h)_{p},\quad h>0.
                     \end{equation*}
\end{theorem}

\bigskip

\section{Improvements of some estimates of approximation by Bernstein operators in H\"older spaces}

Now let us consider  the approximation of functions by polynomial operators for which we cannot apply the methods from Section~4.
These  are, for example, the Bernstein, Kantorovich, and Szasz-Mirakyan polynomial operators. We restrict ourself to the Bernstein operators
\begin{equation*}
B_n(f,x)=\sum_{k=0}^n f\(\frac kn\)\binom{n}{k}x^k (1-x)^{n-k}.
\end{equation*}

Everywhere below $I=[0,1]$, $\vp(x)=\sqrt{x(1-x)}$, $\Vert f\Vert=\Vert f\Vert_\infty$, $\w_r^\vp(f,h)=\w_r^\vp(f,h)_\infty$, $\t_{r,\a}^\vp(f,h)=\t_{r,\a}^\vp(f,h)_\infty$, and $H^{r,\a,\vp}=H_\infty^{r,\a,\vp}$.

It is well-known (see~\cite{To94}) that for any $f\in C(I)$ the following two-sided estimate holds:
\begin{equation}\label{eqber1}
  \Vert f-B_n(f)\Vert\asymp \w_2^\vp \(f,\frac1{\sqrt n}\).
\end{equation}
In~\cite{BR} it was proved that for any $f\in C(I)$ and $0<\a<2$
\begin{equation}\label{eqber2}
  \Vert f-B_n(f)\Vert_{H^{2,\a,\vp}}\le C\t_{2,\a}^\vp \(f,\frac1{\sqrt n}\),
\end{equation}
where $C$ is a constant independent of $f$ and $n$.

Clearly, inequality (\ref{eqber2}) is not sharp. Indeed, if $f^{(r-1)}$ is locally absolutely continuous and $\vp^r f^{(r)}\in C(I)$, then
\begin{equation}\label{eqberder}
  \w_r^\vp(f,h)\le C_{r} h^r\Vert \vp^r f^{(r)}\Vert
\end{equation}
(see~\cite[p. 63]{DT}). Moreover,  the following result was proved in~\cite{GPTZ}:
\begin{lemma}\label{eqberlem1}
  Let $f\in C^k(I)$ and $k=1,2,\dots,n-1$. Then
\begin{equation}\label{eqberlem1eeq}
  \Vert (f-B_n(f))^{(k)}\Vert\le \Vert f^{(k)}-B_{n-k}(f^{(k)})\Vert+\min\left\{1,\frac{(k-1)^2}{n}\right\}\Vert f^{(k)}\Vert+\w_1\(f^{(k)},\frac kn\),
\end{equation}
where $\w_1(f,h)=\sup\{|f(x)-f(y)|,\, |x-y|<h,\, x,y\in I\}$.
\end{lemma}

Thus, by (\ref{eqberder}), (\ref{eqberlem1eeq}), and (\ref{eqber1}),  we obtain for $f\in C^2(I)$, $n\ge 3$, and $0<\a\le 2$
\begin{equation}\label{eqber3}
\begin{split}
\sup_{0<h<1}\frac{\w_2^\vp(f-B_n(f),h)}{h^\a}&\le  C\Vert \vp^2(f''-B_n''(f))\Vert\le C\Vert f''-B_n''(f)\Vert\\
&\le C\(\Vert f'' -B_{n-2}(f'')\Vert+\w_1\(f'',\frac2n\)+\frac1n\Vert f''\Vert\)\\
&\le C\(\w_2^\vp\(f'',\frac1{\sqrt {n-2}}\)+\w_1\(f'',\frac2n\)+\frac1n\Vert f''\Vert\).
\end{split}
\end{equation}
If in addition, $f\in C^4(I)$, then from (\ref{eqber3}) we get
\begin{equation*}
\Vert f-B_n(f)\Vert_{H^{2,\a,\vp}}=\mathcal{O}\(\frac1n\), \quad n\rightarrow\infty.
\end{equation*}
At the same time, inequality (\ref{eqber2}) yields the following less sharp estimate:
\begin{equation*}
\Vert f-B_n(f)\Vert_{H^{2,\a,\vp}}=\mathcal{O}\(\frac1{n^{1-\a/2}}\),\quad n\to\infty.
\end{equation*}

Let us present here an improvement of (\ref{eqber2}).
\begin{proposition}\label{propber}
  Let $f\in H^{2,\a,\vp}$, $0<\a\le 2$, $k\ge 2$, and $0\le \g\le 1/2 $. Then
  \begin{equation}\label{eqber4}
    \Vert f-B_n(f)\Vert_{H^{2,\a,\vp}}\le C_k \(\t_{k,\a}^\vp\(f,\frac1{n^\g}\)+n^{\a \g} \w_2^\vp\(f,\frac1{\sqrt n}\)\).
  \end{equation}
\end{proposition}

The proof of this proposition is standard. One should only take into account Lemma~\ref{lemHolderdt}.

Let us show that (\ref{eqber4}) is an improvement of (\ref{eqber2}). Indeed, let $\w_k^\vp(f,h)=\mathcal{O}(h^{\eta_k})$, where $k\ge 2$.
Then, by choosing $\g=\eta_2/\eta_k$, we get
\begin{equation*}
\Vert f-B_n(f)\Vert_{H^{2,\a,\vp}}=\mathcal{O}\(n^{-\frac12\(\eta_2-\a\frac{\eta_2}{\eta_k}\)}\).
\end{equation*}
At the same time, (\ref{eqber2}) provides only $\mathcal{O}(n^{-\frac12(\eta_2-\a)})$.

As one can see, inequality (\ref{eqber4}) represents only a slight improvement of  (\ref{eqber2}). Because, for $f\in C^4(I)$ we already have
$\mathcal{O}(n^{-1})$, but (\ref{eqber4}) even for $f\in C^N(I)$, $N\ge 4$, yields $\mathcal{O}\(n^{-1+\a/N}\)$.

By using a combination of Proposition~\ref{propber} and Lemma~\ref{eqberlem1}, one can obtain stronger results for smooth functions.

Let us consider the more classical H\"older spaces $C^{r,\a}$ with the norm
\begin{equation*}
\Vert f\Vert_{C^{r,\a}}=\Vert f\Vert+\sup_{0<h< 1}\frac{\Vert \D_h^r f\Vert}{h^\a}
\end{equation*}
instead of $H^{2,\a,\vp}$.

\begin{theorem}\label{thBer}
  Let $f\in C^r(I)$, $0<\a\le r$, $r\in \N$, $\g\ge 0$, and $n\ge r^2+1$. Then
  \begin{equation*}
    \begin{split}
        \Vert f-B_n(f)\Vert_{C^{r,\a}}\le \frac1{n^{\g(r-\a)}}\bigg(C\w_2^\vp &\(f^{(r)},\frac1{\sqrt n}\)+\w_1\(f^{(r)},\frac rn\)\\
        &+\frac{(r-1)^2}{n}\Vert f^{(r)}\Vert\bigg)+Cn^{\g \a} \w_2^\vp\(f,\frac1{\sqrt n}\),
     \end{split}
  \end{equation*}
where $C$ is a constant independent of $f$ and $n$ and $\w_1$ was defined in Lemma~\ref{eqberlem1}.
\end{theorem}

\begin{proof}
  We have
  \begin{equation}\label{eqthBer1}
    \sup_{0<h< 1}\frac{\Vert \D_h^r (f-B_n(f))\Vert}{h^\a}\le \(\sup_{0<h< 1/n^\g}+\sup_{1/n^\g\leq h<1}\) \frac{\Vert \D_h^r (f-B_n(f))\Vert}{h^\a}=S_1+S_2.
  \end{equation}
By (\ref{eqber1}), it is easy to see that
  \begin{equation}\label{eqthBer2}
S_2\le Cn^{\g \a}\Vert f-B_n(f)\Vert\le C n^{\g \a}\w_2^\vp\(f,\frac1{\sqrt n}\).
  \end{equation}
By using  (\ref{eqberder}) with $\vp\equiv 1$, Lemma~\ref{eqberlem1}, and (\ref{eqber1}), we obtain
\begin{equation}\label{eqthBer3}
\begin{split}
  S_1&\le \frac {C}{n^{\g(r-\a)}}\Vert (f-B_n(f))^{(r)}\Vert \\
  &\le \frac{C}{n^{\g(r-\a)}}\(\w_2^\vp \(f^{(r)},\frac1{\sqrt n}\)+\w_1\(f^{(r)},\frac rn\)
        +\frac{(r-1)^2}{n}\Vert f^{(r)}\Vert\).
\end{split}
\end{equation}
  Thus, combining (\ref{eqthBer1}), (\ref{eqthBer2}), and (\ref{eqthBer3}), we have proved the theorem.
\end{proof}

In this sense Theorem~\ref{thBer} is also an improvement of the main results of~\cite{GPTZ} and~\cite{GPT}.

\bigskip

\textbf{Acknowledgements.} We would like to thank the referees for their valuable comments and remarks.
The authors were supported by FP7-People-2011-IRSES Project number 295164  (EUMLS: EU-Ukrainian Mathematicians for Life Sciences).
The first author was partially supported by Project 15-1bb$\setminus$19, ''Metric Spaces, Harmonic Analysis of
Functions and Operators and Singular and Non-classic Problems for Differential
Equations'', Donetsk National University (Vinnitsa, Ukraine).


\end{document}